\documentclass{amsart}

\usepackage{amsfonts}
\usepackage{amsmath}
\usepackage{amssymb}

\newtheorem{theorem}{Theorem}[section]

\newtheorem{corollary}[theorem]{Corollary}
\newtheorem{definition}[theorem]{Definition}
\newtheorem{example}[theorem]{Example}
\newtheorem{lemma}[theorem]{Lemma}
\newtheorem{proposition}[theorem]{Proposition}

\begin{document}

\title{Canonical variation of a Lorentzian metric}
\thanks{This paper was supported in part by Junta de Andaluc\'ia research FQM-324.}
\author{Benjam\'in Olea}

\address{Departamento de Matem\'aticas\\Instituto E.S. Torre Almenara\\Mijas. M\'alaga. Spain.}

\email{benji@agt.cie.uma.es}
\thanks{}

\begin{abstract}

 Given a Lorentzian manifold $(M,g_L)$ and a timelike unitary vector field $E$, we can construct the Riemannian
metric $g_R=g_L+2\omega\otimes\omega$, being $\omega$ the metrically equivalent one form to $E$.
We relate the curvature of both metrics,
especially in the case of $E$ being Killing or closed, and 
we use the relations obtained to give some results about $(M,g_L)$.
\end{abstract}

\maketitle

\section{Introduction}

Given a Lorentzian manifold $(M,g_L)$ and a timelike unitary vector field $E\in\mathfrak{X}(M)$ we can construct the Riemannian metric
\begin{equation}
g_R=g_L+2\omega\otimes\omega,\label{cambiocanonico}
\end{equation}
being $\omega$  the $g_L$-metrically equivalent one form to $E$. 
This construction is frequently used to exploit the positive definiteness of $g_R$, which provides some conclusions about $g_L$. 
For example, 
in \cite{Flores} it is used to prove, under suitable conditions,
 the existence of periodic timelike geodesics in a compact Lorentzian manifold
and in \cite{Bejancu} to give a Bernstein theorem for lightlike hypersurfaces  in $\mathbb{R}^n_1$.
  A similar construction has also been used to induce a Riemannian metric on a lightlike hypersurface of a Lorentzian manifold, which allows to define its 
extrinsic scalar curvature, \cite{Antidogbe}.
Some aspects of Lorentzian metrics constructed in this way haven been studied in \cite{Sharma1} and \cite{Sharma2}.

The dual construction of (\ref{cambiocanonico}), i.e., given a Riemannian manifold $(M,g_R)$ define 
the Lorentzian metric
\begin{equation}
g_L=g_R-2\omega\otimes\omega, \label{cambiocanonico2}
\end{equation}
 is also 
interesting because it provides some important examples of Lorentzian manifold, \cite{Hawking,Gut2,Yurtsever}.

In this paper,
 we consider a Riemannian or a Lorentzian manifold $(M,g)$ and a (timelike in the Lorentzian case) unitary vector field $E\in\mathfrak{X}(M)$. We call
$\varepsilon=g(E,E)$  and define, for $t\neq-\varepsilon$,
\begin{equation*}
 g_t=g+t\omega\otimes\omega.
\end{equation*}
This metric is Lorentzian if $t<-\varepsilon$ and Riemannian if $-\varepsilon<t$. We call it 
\textit{the canonical variation of $g$ along $E$} due to
its analogy with the canonical variation of a Riemannian submersion, where the metric of the fibres is multiplied by a parameter 
$t$, \cite{Besse}.
In \cite{Strake}, the metric $g_t$ is also called variation of $g$ and it is used to contruct a Riemann metric with strictly
positive
sectional curvature from another Riemann metric with nonnegative sectional curvature.

Obviously, the construction (\ref{cambiocanonico}) corresponds to $t=2$ and $\varepsilon=-1$ and (\ref{cambiocanonico2}) to $t=-2$ and $\varepsilon=1$ 
and we call them \textit{standard canonical variation}.

Observe that metrics $g_R$ and $g_L$ in (\ref{cambiocanonico}) and (\ref{cambiocanonico2}) are related in the same way and therefore it would be sufficient  just to
study,
for example, the metric $g_R$ constructed in (\ref{cambiocanonico}) to obtain analog results for $g_L$ constructed in (\ref{cambiocanonico2}). Nevertheless, the
introduction of the parameters $t$ and $\varepsilon$ allows us to handle jointly (\ref{cambiocanonico}) and (\ref{cambiocanonico2}), obtaining
formulas  easily adaptable to each case.

We use $X,Y,Z$ letters for othogonal
vector fields to $E$ and $U,V,W$ for arbitrary vector fields.
We write the geometric objects derived from $g_t$ with a $t$ subscript, except for the connection
which will be denoted by $\nabla^t$. For example $K_t$, $Ric_t$ and $S_t$ are
the sectional, Ricci and scalar curvature  of $g_t$ respectively. When we deal with $g_L$, $g_R$ or $g$ we use 
a $L$, $R$ or no subscript respectively. 

In the second section, we show a formula relating the difference tensor of the connections $\nabla^t$ and $\nabla$,  the exterior differential of $\omega$
and the Lie derivative  $L_E g$. We also introduce the notion of vector field 
with  \textit{normal associated endomorphism}, which extends the notion
 of closed or conformal vector field. This concept will be useful to simplify the computation of the 
curvature of the canonical variation that are made in the third section. 
In section \ref{seccionKilling} we consider the standard canonical variation along a Killing unitary vector field
and we obtain some inequalities about the curvature. 
We use the Berger theorem to show that if there exists a timelike unitary Killing vector field in a compact Lorentzian manifold with negative sectional curvature on timelike planes, 
then it has odd dimension. We also use Bochner techniques to give an integral inequality in a compact Lorentzian manifold furnished with a timelike Killing vector field.
In the fifth section we give a result  about the geodesic completeness of the 
canonical variation along a closed vector field and in the last section we consider how a lightlike hypersurface is transformed under standard canonical variation.
We also give some sufficient conditions for a compact lightlike hypersurface to be totally geodesic.

\section{Preliminaries}

From now on, let $(M,g)$ be a Riemannian or a Lorentzian manifold, 
$E\in\mathfrak{X}(M)$  a (timelike in the Lorentzian case) unitary
 vector field, $\omega$ its metrically equivalent one form and  $\varepsilon=g(E,E)$.

\begin{definition} Fixed $t\in\mathbb{R}-\{-\varepsilon\}$, the canonical variation of $g$ along $E$ is defined as
\begin{equation}\label{metricagt}
g_t=g+t\omega\otimes\omega.
\end{equation}
If $t=-2\varepsilon$, then it is called the standard canonical variation of $g$.
\end{definition}

\begin{example}
We give some examples of the standard canonical variation of a Riemannian manifold.
 \begin{enumerate}
  \item The standard canonical variation along any parallel vector field in the Euclidean space gives us
the Minkowski space.

\item Consider the hyperbolic space $\mathbb{H}^n=\left(\mathbb{R}^n,dx_1^2+e^{2x_1}\sum_{i=2}^ndx_i^2\right)$. The standard canonical variation along $E=\partial_{x_1}$
gives us $-dx_1^2+e^{2x_1}\sum_{i=2}^ndx_i^2$, which is a piece of $\mathbb{S}^n_1$.
On the other hand, if we take $E=e^{-x_1}\partial_{x_n}$, the standard canonical variation is $dx_1^2+e^{2x_1}\left(\sum_{i=2}^{n-1}dx_i^2-dx_n^2\right)$, which is a 
piece of $\mathbb{H}^n_1$.

\item Fix $v\in\mathbb{R}^n$ and let $f:\mathbb{S}^{n-1}\rightarrow\mathbb{R}$ be given by $f(p)=p\cdot v$.
Take $\mathbb{S}_+^{n-1}=\{p\in\mathbb{S}^{n-1}:f(p)>0\}$ with its induced metric $g_0$. Since 
$Hess_f=-f g_0$, we have that  $\left(\mathbb{S}_+^{n-1}\times\mathbb{R},g_0+f^2 dt^2\right)$
is an open set of $\mathbb{S}^n$. 
The standard canonical variation along $E=\frac{1}{f}\partial_t$ is the piece of $\mathbb{S}^n_1$ given by $\left(\mathbb{S}_+^{n-1}\times\mathbb{R},g_0-f^2 dt^2\right)$.

\item The Lorentzian Berger spheres are obtained as the
standard canonical variation along the Hopf vector field of the Euclidean spheres $\mathbb{S}^{2n+1}$.

\item The Lie group $SL_2(\mathbb{R})$ can be furnished with the bi-invariant Lorentzian metric given by
$<X,Y>_L=\frac{1}{2}tr(XY)$ or with the Riemannian left-invariant metric given by $<X,Y>_R=\frac{1}{2}tr(XY^t)$, being 
$X,Y\in\mathfrak{sl}_2(\mathbb{R})$. If we call

\[ E=\left( \begin{array}{cc}
\phantom{-}0 & 1 \\
-1 & 0
\end{array} \right)\in\mathfrak{sl}_2(\mathbb{R}),
\]
then it is easy to show that $<>_R$ is the standard canonical variation of $<>_L$ along $E$.
Recall that $\left(SL_2(\mathbb{R}),<>_L\right)$ is isometric to $\mathbb{H}_1^3$.
\end{enumerate}
\end{example}

We call $D^t=\nabla^t-\nabla$ the difference tensor.
We can relate this tensor, $L_E g$ and $\omega$ as follows.

\begin{proposition}\label{diferencianablas} Given 
$U,V,W\in\mathfrak{X}(M)$, it holds
\begin{equation}
g_t(D^t(U,V),W)=\frac{t}{2}\Big(\omega(W)\left(L_E g\right)(U,V)+\omega(U)d\omega(V,W)+\omega(V)d\omega
(U,W)\Big).\label{eqdiferencianablas}
\end{equation}
\end{proposition}
\begin{proof} We can suppose that Lie brackets vanish. By the Koszul formula and
equation (\ref{metricagt}) we have
\begin{eqnarray*}
2 g_t(\nabla^t_{U}V,W)
&=&2g(\nabla_U V,W)+t\big(U\big(\omega\otimes\omega(V,W)\big)+V\big(\omega\otimes\omega(U,W)\big)\\
&-&W\big(\omega\otimes\omega(U,V)\big)\\
&=&2g(\nabla_U
V,W)+t\big(\nabla_U(\omega\otimes\omega)(V,W)+\nabla_V(\omega\otimes\omega)(U,W)\\
&-&\nabla_W(\omega\otimes\omega)(U,V)\big)+2t\omega\otimes\omega(W,\nabla_U V).
\end{eqnarray*}
Using again equation (\ref{metricagt}),
\begin{eqnarray*}
 g_t(D^t(U,V),W)=\frac{t}{2}\Big(\nabla_U(\omega\otimes\omega)(V,W)+\nabla_V(\omega\otimes\omega)(U,W)-\nabla_W(\omega\otimes\omega)(U,V)\Big).
\end{eqnarray*}
Now, since 
$d\omega(U,V)=\left(\nabla_{U}\omega\right)(V)-\left(\nabla_{V}\omega\right)(U)$ and
$\left(L_E g\right)(U,V)=\left(\nabla_{U}\omega\right)(V)+\left(\nabla_{V}\omega\right)(U)$, 
 the right hand of the above expression is
\begin{eqnarray*}
\frac{t}{2}\Big(\omega(W)\left(L_E g\right)(U,V)+\omega(V)d\omega(U,W)+\omega(U)d\omega(V,W)\Big).
\end{eqnarray*}
\end{proof}

We have the following  consequences.

\begin{corollary}\label{corolario1} Take $V,W,X,Y\in\mathfrak{X}(M)$ with $X,Y\in E^\perp$.
\begin{enumerate}
\item $g(D^t(X,V),X)=0$.
\item $g(D^t(V,E),E)=0$.
\item $g_t(D^t(V,E),W)+g_t(D^t(W,E),V)=t\big(\omega(V)g(W,\nabla_EE)+\omega(W)g(V,\nabla_EE)\big)$.
\item $\nabla^t_{E}E=(1+\varepsilon t)\nabla_{E}E$.
\item $D^t(X,Y)=\frac{t}{2(1+\varepsilon t)}\left(L_Eg\right)(X,Y)E$.
\item $\left(L_Eg_t\right)(X,Y)=\left(L_Eg\right)(X,Y)$. In
particular, if $E$ is orthogonally conformal for $g$, then it is also orthogonally
conformal for $g_t$. 
\item $div_t V=div V$.
\end{enumerate}
\end{corollary}

We need to introduce the following concepts for next sections.

\begin{definition}
The associated endomorphism to a vector field $U\in\mathfrak{X}(M)$ is
$A_U:\mathfrak{X}(M)\rightarrow \mathfrak{X}(M)$ given
 by $A_U(V)=\nabla_VU$.
\end{definition} 

Let $A_U^*$ be the adjoint endomorphism of $A_U$.
If  
$g(A_U(V),A_U(V))=g(A_U^*(V),A_U^*(V))$ for all $V\in\mathfrak{X}(M)$, or 
equivalently, $A_U$ and $A^*_U$ commutes, then $A_U$ is called normal.

\begin{definition}
Let $E\in\mathfrak{X}(M)$ be a unitary vector field. $A_E$
 is orthogonally normal if
$g(A_E(X),A_E(X))=g(A^{*\perp}_E(X),A^{*\perp}_E(X))$ for all $X\in E^\perp$, 
being $A^{*\perp}_E(X)$ the orthogonal component to $E$.
\end{definition}

If $A_E$ is orthogonally normal, then it is easy to show that   $A_E$ is normal if and only
if $E$ is geodesic. Moreover, if $U\in\mathfrak{X}(M)$ with $A_U$ normal and  $E$ is its unitary, then 
$A_E$ is orthogonally normal.

\begin{example}\label{ejemplocamponormal}We give some examples of vector field with normal associated endomorphism.
 \begin{enumerate}
  \item The associated endomorphism of a closed vector field is normal.
\item If $U$ is a  conformal vector field, then   $A^*_U=2\rho\cdot id-A_U$ for certain $\rho\in C^\infty(M)$ and thus it
is straightforward
to check that $A_U$ is normal.

\item $A_U$ is normal if and only if its associated matrix respect to
 a frame field is normal (it commutes with its transpose). Using this, it is easy to check that 
$U=(x+y)\partial_x+(y+z)\partial_y+(x+z)\partial_z$ has normal associated endomorphism in $\mathbb{R}^3$ and so
its unitary $E$ has orthogonally normal associated endomorphism in $\mathbb{R}^3-\{0\}$.
Observe that $U$ is not closed neither conformal.

\item Let $G$ be a Lie group with a bi-invariant metric $g$. If  $U$ is any left-invariant vector field, then $A^*_U(X)=-A_U(X)$ and hence it is normal.
\end{enumerate}
\end{example}

\begin{proposition} Let $(M,g)$ be a semi-Riemannian manifold, $U\in\mathfrak{X}(M)$ a  vector field with normal associated endomorphism
 and
$\overline{M}$ a nondegenerate hypersurface of $M$ with unitary normal $N$. The projection of $U$ onto $\overline{M}$
 has normal associated endomorphism
 if and only if
\begin{align}
 g(A_U(X),N)^2-g(X,A_U(N))^2=2g(U,N)\big(g( A_U(X),S(X))-g(A_U( S(X)),X)\big)\label{ecuacionproyeccionnormal}
\end{align}
 for all $X\in\mathfrak{X}(\overline{M})$, being $S$ the shape operator  of $\overline{M}$.
\end{proposition}
\begin{proof}
Take
 $\delta=g(N,N)=\pm1$ and write  $U=V+\delta g(U,N)N$ with $V\in\mathfrak{X}(\overline{M})$. Call
$A_V(X)=tan(\nabla_XV)$, where $tan(\cdot)$ denotes the projection onto $\overline{M}$.
Given $X\in\mathfrak{X}(\overline{M})$,
\begin{align*}
A_V(X)&=tan\left(A_U(X)\right)-\delta g(U,N)\nabla_XN,\\
A^*_V(X)&=tan\left(A^*_U(X)\right)-\delta g(U,N)\nabla_XN.
\end{align*}
Therefore, 
\begin{align*}
 g(A_V(X),A_V(X))&=g(tan(A_U(X)),tan(A_U(X)))-2\delta g(U,N)g(\nabla_XN,A_U(X))\\
&+g(U,N)^2g(\nabla_XN,\nabla_XN)\\
&=g(A_U(X),A_U(X))-\delta g(A_U(X),N)^2-2\delta g(U,N)g(\nabla_XN,A_U(X))\\
&+g(U,N)^2g(\nabla_XN,\nabla_XN).
\end{align*}
Analogously,
\begin{align*}
  g(A^*_V(X),A^*_V(X))&=g(A^*_U(X),A^*_U(X))-\delta g(A^*_U(X),N)^2-2\delta g(U,N)g(\nabla_XN,A^*_U(X))\\
&+g(U,N)^2g(\nabla_XN,\nabla_XN).
\end{align*}
Hence, $A_V$ is normal if and only if equation (\ref{ecuacionproyeccionnormal}) holds.
\end{proof}

\begin{example}Above proposition provides more  examples of vector fields with normal associated endomorphism.
 \begin{enumerate} 
\item If $U$ is  tangent to $\overline{M}$ and it has normal associated endomorphism,
 then the restriction of $U$  to $\overline{M}$ also has normal associated endomorphism since equation (\ref{ecuacionproyeccionnormal}) holds trivially.
  \item The projection of a Killing vector field onto an umbilic hypersurface has normal associated endomorphism, since 
equation (\ref{ecuacionproyeccionnormal}) holds in this case. Observe that the projection of a Killing vector field
onto a hypersurface is also a Killing vector field if and only if 
the hypersurface is totally geodesic.
 \end{enumerate}
\end{example}

\section{Curvature of the canonical variation}

To relate the curvature tensor $R$ of a metric $g$ and the curvature tensor $R^t$ of its canonical variation $g_t$ we need the following general lemma.

\begin{lemma}\label{lemadiferenciacurvatura} Let $\nabla$ and $\nabla^t$ be two arbitrary connections on a manifold $M$ with curvature tensors $R$ and $R^t$ 
respectively. Given
$U,V,W\in\mathfrak{X}(M)$ we have
\begin{eqnarray}
R^t_{UV}W&=&R_{UV}W+(\nabla_{U}D^t)(V,W)-(\nabla_{V}D^t)(U,W)\nonumber\\
&+&D^t(U,D^t(V,W))-D^t(V,D^t(U,W)).\label{diferenciaR}
\end{eqnarray}
\end{lemma}

\begin{theorem}\label{curvatura1} Let $(M,g)$ be a  Riemannian or a Lorentzian manifold and $g_t$
the canonical variation along a (timelike) unitary 
vector field $E$ with $A_E$ orthogonally normal. Given $X\in E^\perp$
\begin{eqnarray*}
g_t(R^t_{XE}E,X)&=&g(R_{XE}E,X)+t\left(\varepsilon g(\nabla_XA_E(E),X)-g(A_E(
E),X)^2\right)\\
&+&\frac{t(2\varepsilon+t)}{2}\left(g(A_E(X),A_E(X))-g(A^2_E(X),X)\right).
\end{eqnarray*}
\end{theorem}

\begin{proof} Applying equation (\ref{diferenciaR}), we have
\begin{eqnarray*}
g(R^t_{XE}E-R_{XE}E,X)=g((\nabla_X D^t)(E,E),X)-g((\nabla_E
D^t)(X,E),X)\\
+g(D^t(X,D^t(E,E)),X)-g(D^t(E,D^t(X,E)),X).
\end{eqnarray*}
We compute each term applying proposition \ref{diferencianablas}. The first one is $g((\nabla_X D^t)(E,E),X)=\varepsilon tg(\nabla_X \nabla_E E,X)-\varepsilon td\omega(\nabla_X E,X)$.
For the second one we  use corollary \ref{corolario1}.
\begin{eqnarray*}
g((\nabla_E D^t)(X,E),X)
&=&-g_t(D^t(X,E),\nabla_E X)-g_t(D^t(\nabla_E X,E),X)\\
&=&-tg(E,\nabla_E X)g(\nabla_E E,X)\\&=&tg(\nabla_E E,X)^2.
\end{eqnarray*}
The third one is zero since $D^t(E,E)\perp E$. The last one.
\begin{eqnarray*}
g(D^t(E,D^t(X,E)),X)&=&\frac{\varepsilon t}{2}d\omega(D^t(X,E),X)\\
&=&\frac{\varepsilon t}{2}\left(g(\nabla_{D^t(X,E)}E,X)-g_t(D^t(X,E),\nabla_X E)\right)\\
&=&\frac{\varepsilon t}{2}\left(g(\nabla_{D^t(X,E)}E,X)-\frac{\varepsilon
t}{2}d\omega(X,\nabla_X E)\right).
\end{eqnarray*}

Now, we have
$d\omega(A_E(X),X)=g(A^2_E(X),X)-g(A_E(X),A_E(X))$
and 
\begin{equation*}
g(A_E(D^t(X,E)),X)=\frac{\varepsilon t}{2}\Big(g(A_E(X),A^{*\perp}_E(X))-g\big(A^{*\perp}_E(X),A^{*\perp}_E(X)\big)\Big).
\end{equation*}
Therefore, since $A_E$ is orthogonally normal, 
$\left(\varepsilon+\frac{t}{4}\right)d\omega(\nabla_X
E,X)+\frac{\varepsilon}{2}g(\nabla_{D^t(X,E)}E,X)=
\frac{2\varepsilon+t}{2}\left(g(A^2_E(X)),X)-g(A_E(X),A_E(X))\right)$
and
we obtain the desired result.
\end{proof}

\begin{corollary} Let $(M,g)$ be a  Riemannian or Lorentzian  surface and $g_t$
the canonical variation of $g$ along a (timelike) unitary and geodesic vector field $E$. If $K$ and $K^t$ denote the Gauss curvature
of $g$ and $g_t$ respectively, then $K^t=\frac{1}{1+\varepsilon t}K$.
\end{corollary}

\begin{corollary} Let $(M,g_R)$ be a Riemannian manifold and $g_t$
the canonical variation along a  unitary vector field $E$ with $A_E$ normal.
If $\Pi$ is a plane containing $E$, then
\begin{eqnarray*}
K_t(\Pi)&\leq&\frac{1}{1+t}K_R(\Pi)\ for\ t\in (-\infty,-2)\cup (-1,0),\\
K_t(\Pi)&\geq&\frac{1}{1+t}K_R(\Pi)\ for\ t\in (-2,-1)\cup (0,\infty).
\end{eqnarray*}
In particular, if $g_L$ is the  
standard canonical variation, then $K_{L}(\Pi)=-K_R(\Pi)$.
\end{corollary}
\begin{proof} Since $A_E(X)\perp E$, we have 
\begin{eqnarray*}
g(A^2_E(X),X)\leq
\sqrt{g(A_E(X),A_E(X))}\sqrt{g(A^{*\perp}_E(X),A^{*\perp}_E(X))}=g(A_E(X),A_E(X)).
\end{eqnarray*}
Using that $E$ is geodesic and theorem \ref{curvatura1} we get the result.
\end{proof}

\begin{corollary}\label{formulaRicci} Let $(M,g)$ be a Lorentzian or Riemannian
manifold and $g_t$ the canonical variation along a (timelike) unitary vector field $E$
with $A_E$ orthogonally normal. Then
\begin{eqnarray*}
Ric_t(E,E)=Ric(E,E)+\varepsilon t\, div\nabla_E
E+\frac{t(2\varepsilon+t)}{2}(||A'_E||^2-tr(A'^2_E)),
\end{eqnarray*}
where  $A'_E$ is the restriction to $E^\perp$ of $A_E$. In particular, if $(M,g_R)$ is a Riemannian compact manifold,
\begin{eqnarray*}
 \int_{M}Ric_t(E,E)dg_t&\leq& \sqrt{|1+t|}\int_{M} Ric(E,E)dg_R \ for\ t\in (-\infty,-2)\cup (0,\infty),\\
\int_{M}Ric_t(E,E)dg_t&\geq& \sqrt{|1+t|}\int_{M} Ric(E,E)dg_R\ for\ t\in (-2,-1)
\end{eqnarray*}
and if $g_L$ is the standard canonical variation, then 
$\int_{M}Ric_{L}(E,E)dg_L= \int_{M} Ric(E,E)dg_R$.
\end{corollary}
\begin{proof} From theorem \ref{curvatura1} we have
\begin{eqnarray*}
Ric_t(E,E)&=&Ric(E,E)+t\big(\varepsilon( div\,\nabla_E E-\varepsilon g(\nabla_E\nabla_E
E,E))-g(\nabla_E E,\nabla_E E)\big)
\\&&+\frac{t(2\varepsilon+t)}{2}\left(||A'_E||^2-tr(A'^2_E)\right) \\
&=&Ric(E,E)+\varepsilon t  div\,\nabla_E E+\frac{t(2\varepsilon+t)}{2}\left(||A'_E||^2-tr(A'^2_E)\right).
\end{eqnarray*}

For the second part, we take into account that the canonical volume forms of  $g_t$ and $g$ are related by
 $dg_t=\sqrt{|1+t|}dg_R$ and that $||A'_E||^2-tr(A'^2_E)\geq0$ because $A_E$ is orthogonally normal.
\end{proof}

\begin{theorem}\label{curvatura2} Let $(M,g)$ be a Riemannian or Lorentzian manifold and $g_t$ the canonical
 variation along a (timelike) unitary vector field $E$. If $X,Y\in E^\perp$, then
\begin{eqnarray*}
g_t(R^t_{XY}Y,X)&=&g(R_{XY}Y,X)\\
&+&\frac{t}{1+\varepsilon
t}\Big(g(A_E(X),X)g(A_E(Y),Y)-g(A_E(X),Y)g(A_E(Y),X)\\
&-&\frac{4+3\varepsilon
t}{4}d\omega(X,Y)^2\Big).
\end{eqnarray*}
\end{theorem}
\begin{proof} From equation (\ref{diferencianablas}),
\begin{eqnarray*}
g(R^t_{XY}Y-R_{XY}Y,X)&=&
g((\nabla_{X}D^t)(Y,Y),X)-g((\nabla_{Y}D^t)(X,Y),X)\\
&+&g(D^t(X,D^t(Y,Y)),X)-g(D^t(Y,D^t(X,Y)),X).
\end{eqnarray*}
We compute each term using proposition \ref{diferencianablas}. Ther fisrt one.
\begin{eqnarray*}
g((\nabla_{X}D^t)(Y,Y),X)&=&
-\frac{t}{2(1+\varepsilon t)}\left(L_E g\right)(Y,Y)g(E,\nabla_X X)-tg(E,\nabla_X Y)d\omega(Y,X)\\
&=&\frac{t}{1+\varepsilon t}g(\nabla_Y E,Y)g(X,\nabla_X E)+tg(Y,\nabla_X
E)g(\nabla_Y E,X)\\
&-&tg(\nabla_X E,Y)^2.
\end{eqnarray*}
The second term.
\begin{eqnarray*}
g((\nabla_{Y}D^t)(X,Y),X)&=&
-\frac{t}{2(1+\varepsilon t)}\left(L_E g\right)(X,Y)g(E,\nabla_Y
X)-\frac{t}{2}g(E,\nabla_Y
X)d\omega(Y,X)\\
&=&\frac{t}{2}g(\nabla_Y E,X)\left(\frac{1}{1+\varepsilon t}\left(L_E g\right)(X,Y)+d\omega(Y,X)\right)\\
&=&\frac{t(2+\varepsilon t)}{2(1+\varepsilon t)}g(\nabla_Y E,X)^2-\frac{\varepsilon t^2}{2(1+\varepsilon t)}g(\nabla_X E,Y)g(\nabla_Y E,X).
\end{eqnarray*}
The third one vanishes. Using corollary \ref{corolario1}, the last one is 
$g(D^t(Y,D^t(X,Y)),X)=\frac{\varepsilon t^2}{4(1+\varepsilon t)}\left(g(\nabla_Y
E,X)^2-g(Y,\nabla_X E)^2\right)$.

Now,
\begin{eqnarray*}
&&g_t(R^t_{XY}Y,X)-g(R_{XY}Y,X)=\frac{t}{1+\varepsilon t}g(\nabla_X
E,X)g(\nabla_Y E,Y)\\
&&+\frac{t}{2(1+\varepsilon t)}\left((2+3\varepsilon t)g(\nabla_X E,Y)g(\nabla_Y
E,X)-\frac{4+3\varepsilon t}{2}(g(\nabla_X E,Y)^2+g(\nabla_Y E,X)^2)\right)\\
&&=\frac{t}{1+\varepsilon t}\left(g(\nabla_X E,X)g(\nabla_Y E,Y)-g(\nabla_X
E,Y)g(\nabla_Y E,X)-\frac{4+3\varepsilon t}{4}d\omega(X,Y)^2\right).
\end{eqnarray*}
\end{proof}

\begin{example} A Riemannian manifold $(M,g_R)$ is called of quasi-constant sectional curvature if there exists a unitary vector field $E\in\mathfrak{X}(M)$
such that the sectional curvature of any plane
 only depends on the basepoint and the angle between the plane and $E$, \cite{Popescu,Ganchev}.
 On the other hand, a Lorentzian manifold $(M,g_L)$ is called infinitesimal null isotropy if there exists a timelike unitary vector 
 field $E\in\mathfrak{X}(M)$ such that the lightlike sectional curvature respect to $E$ only depends on the basepoint, \cite{Harris,Karcher}. 
 After theorems \ref{curvatura1} and \ref{curvatura2} we can easily check both definitions are equivalent in their respective settings.
 
 Indeed. Suppose that $(M,g_R)$ is a Riemannian manifold of quasi-constant sectional curvature. If we call $\omega$ and $\theta$ the one forms 
 metrically equivalent to $E$ and $\nabla_EE$ respectively, then it is known that  $\nabla^R_X E=\lambda X$, $d\omega(X,Y)=0$ and 
 $\left(\nabla^R_X\theta\right)(Y)-\theta(X)\theta(Y)=\gamma g_R(X,Y)$, being $X,Y\in E^\perp$ and $\lambda,\gamma$ functions on $M$, \cite{Ganchev}.
Take $g_L$ the standard canonical variation along $E$. Using above formulas
and theorems \ref{curvatura1} and \ref{curvatura2}, it is easy to show that, in $(M,g_L)$, the sectional curvature 
of planes orthogonal to $E$ and planes containing $E$
only depends on the basepoint, which implies that it  is infinitesimal null isotropy.
In the same manner, we can check that if $(M,g_L)$ is a Lorentzian manifold infinitesimal null isotropy, then the standard canonical variation is 
a Riemannian manifold of quasi-constant sectional curvature.
\end{example}

\begin{corollary}\label{Ricci2}Let $(M,g)$ be a Riemannian or Lorentzian manifold and $g_t$ the canonical
variation along a (timelike) unitary vector field
$E$ with $A_E$  orthogonally normal. Given  $X\in E^\perp$,
\begin{eqnarray*}
Ric_t(X,X)&=&Ric(X,X)-\frac{t}{1+\varepsilon t} g(R_{XE}E,X)+\frac{t}{1+\varepsilon
t}g(A_E(X),X) div\, E\\
&+&\frac{\varepsilon t^2}{1+\varepsilon t}g(A^2_E(X),X)-tg(A_E(X),A_E(X))
\\&+&\frac{t}{1+\varepsilon t}\left(g(\nabla_X A_E(E)
,X)-\varepsilon g(A_E( E),X)^2\right).
\end{eqnarray*}
\end{corollary}
\begin{proof} Take $\{e_1,\ldots,E_t\}$ an orthonormal basis for $g_t$ with
 $E_t=\frac{1}{\sqrt{|t+\varepsilon|}}E$. 
Using theorem \ref{curvatura1},
\begin{eqnarray*}
g_t(R^t_{E_t X}X,E_t)
&=&\frac{1}{|\varepsilon+t|}g(R_{XE}E,X)+\frac{t}{|\varepsilon+t|}\left(\varepsilon
g(\nabla_X\nabla_E E,X)-g(\nabla_E
E,X)^2\right)\\
&+&\frac{t(2\varepsilon+t)}{2|\varepsilon+t|}\left(g(A_E(X),A_E(X))-g(A^2_E(X),X)\right).
\end{eqnarray*}
On the other hand, from theorem \ref{curvatura2},
\begin{align*}
&\sum_{i=1}^{n-1}g_t(R^t_{e_iX}X,e_i)=\sum_{i=1}^{n-1}g(R_{e_iX}X,e_i)+&\\
&\frac{t}{1+\varepsilon
t}\left(g(A_E(X),X)div\, E-g(X,A^2_E(X))-\frac{4+3\varepsilon t}{4}\sum_{i=1}^{n-1}d\omega(X,e_i)^2\right).&\\
\end{align*}
But now observe that
\begin{eqnarray*}
\sum_{i=1}^{n-1}d\omega(X,e_i)^2&=&g(A_E(X),A_E(X))+g(A^{*\perp}_E(X),A^{*\perp}_E(X))-2g(X,A^2_E(X))\\
&=&2g(A_E(X),A_E(X))-2g(X,A^2_E(X)),
\end{eqnarray*}
where the last equality holds because $A_E$ is orthogonally normal.
Putting all together we obtain the result.
\end{proof}

\begin{corollary}\label{curvaturasescalares}Let $(M,g)$ be a Riemannian or Lorentzian manifold and $g_t$ the canonical variation along 
a  (timelike) unitary vector field
$E$ with $A_E$ orthogonally normal. The scalar curvatures $S_t$ and $S$ are related by
\begin{eqnarray*}
S_t&=&S-\frac{2t}{1+\varepsilon t}Ric(E,E)+\frac{2t}{1+\varepsilon t}div\,\nabla_E
E-tg(\nabla_E E,\nabla_E E)\\
&&+\frac{t}{1+\varepsilon t}\left(tr(A'_E)^2-tr(A'^2_E)\right)+\frac{\varepsilon
t^2}{2(1+\varepsilon t)}\left(tr(A'^2_E)-||A'_E||^2\right),
\end{eqnarray*}
where $A'_E$ is the restriction to $E^\perp$ of $A_E$.
\end{corollary}
\begin{proof} We only have to apply corollaries \ref{formulaRicci} and \ref{Ricci2}.
\end{proof}

\begin{example}
 We can easily construct a Riemannian odd sphere with constant negative scalar curvature. Indeed, if we take $g_t$ the canonical variation along the Hopf vector
 field, since it is unitary and Killing, from above corollary, $S_t=2n(2n+1-t)$, which is negative  for $t$ large enough.
\end{example}

Finally, we compute $g_t(R^t_{EX}X,Y)$, being $X,Y\in E^{\perp}$,
which will be useful later.

\begin{proposition}\label{curvatura3}Let $(M,g)$ be a Riemannian or Lorentzian manifold and $g_t$ the 
canonical variation  along a (timelike) unitary vector field $E$. If $X,Y\in E^\perp$, then
\begin{eqnarray*}
g_t(R^t_{EX}X,Y)&=&g(R_{EX}X,Y)+\frac{t}{2}\Big(-\varepsilon(\nabla_X
d\omega)(X,Y)+g(A_E(X),X)g(A_E(E),Y)\\&&-2g(X,A_E(Y))
g(A_E(E),X)+g(A_E(X),Y)g(A_E(E),X)\Big).
\end{eqnarray*}
\end{proposition}
\begin{proof}From formula  (\ref{diferenciaR}),
\begin{eqnarray*}
g_t(R^t_{EX}X-R_{EX}X,Y)&=&
g((\nabla_E D^t)(X,X),Y)-g((\nabla_X
D^t)(E,X),Y)\\
&+&g(D^t(E,D^t(X,X)),Y)-g(D^t(X,D^t(E,X)),Y).
\end{eqnarray*}
As always, we use proposition \ref{diferencianablas} to compute each term. The first one gives us 
$g((\nabla_E D^t)(X,X),Y)=\frac{t}{1+\varepsilon t}g(A_E(X),X)g(A_E(E),Y)+tg(A_E(E)
,X)d\omega(X,Y)$.
The second one.
\begin{eqnarray*}
g((\nabla_X D^t)(E,X),Y)&=&Xg(D^t(E,X),Y)-g(D^t(E,X),\nabla_X Y)-\frac{\varepsilon
t}{2}d\omega(\nabla_X
X,Y)\\
&-&\frac{t}{2}g(E,\nabla_X X)d\omega(E,Y)\\
&=&\frac{\varepsilon t}{2}X(d\omega(X,Y))-\frac{\varepsilon t}{2}d\omega(X,\nabla_X
Y)+\frac{t}{2}g(\nabla_X E,Y)g(\nabla_E
E,X)\\
&-&\frac{\varepsilon t}{2}d\omega(\nabla_X X,Y)+\frac{t}{2}g(\nabla_X E,X)g(\nabla_E E,Y)\\
&=&\frac{\varepsilon t}{2}(\nabla_X
d\omega)(X,Y)+\frac{t}{2}\Big(g(\nabla_X E,Y)g(\nabla_E
E,X)\\
&+&g(\nabla_X E,X)g(\nabla_E E,Y)\Big).
\end{eqnarray*}
The third is $g(D^t(E,D^t(X,X)),Y)=\frac{\varepsilon
t^2}{1+\varepsilon t}g(\nabla_X E,X)g(\nabla_E E,Y)$ and 
the last one vanishes.
\end{proof}

\section{Standard canonical variation along a Killing vector field}\label{seccionKilling}

Suppose that $E$ is a Killing unitary vector field in a Riemannian manifold $(M,g_R)$
 and consider $g_L=g_R-2\omega\otimes\omega$ the standard canonical variation along it.
In this case, from formula (\ref{eqdiferencianablas}),  we have	
\begin{equation*}
\nabla^L_UV=\nabla^R_UV-2\Big(\omega(U)\nabla^R_V E+\omega(V)\nabla^R_U E\Big).
\end{equation*}
Moreover, from corollary \ref{corolario1}, $E$ is also Killing for $g_L$.

The symmetric respect to $E$ of a vector $v=\alpha E+Y$, being $Y\perp E$,  is $v^*=\alpha E-Y$.
The symmetric respect to $E$ of a plane $\Pi=span(X,v)$, being $X\perp E$,  
 is the plane given by 
$\Pi^*=span(X,v^*)$.

We denote $\mathcal{K}^E_L(\Pi)$ the lightlike 
sectional curvature of a lightlike plane $\Pi$ of $(M,g_L)$ respect to $E$.

\begin{proposition}\label{desigualdad1}
 Let $(M,g_R)$ be a Riemannian manifold, $E\in\mathfrak{X}(M)$ a Killing unitary vector field
 and $g_{L}$ the standard canonical  variation along $E$.
\begin{enumerate}
 \item If $\Pi$ is a nondegenerate plane for $g_L$, then $K_R(\Pi^*)\leq-\cos(2\theta)K_L(\Pi)$, 
being $\theta$ the angle between $\Pi$ and $E$.
Moreover, the equality holds if and only if $E\in \Pi$.
\item If $\Pi$ is a degenerate  plane for $g_L$, then $2 K_R(\Pi^*)\leq \mathcal{K}_L^E(\Pi)$.
\item Given $v\in TM$ it holds $ Ric_R(v,v)\leq Ric_L(v^*,v^*)$ and the equality holds if and only if $v$  is proportional to $E$.
\item The scalar curvatures $S_R$ and $S_L$ hold $S_R\leq S_L$ and the  equality holds if and only if $E$ is parallel.
\end{enumerate}
\end{proposition}
\begin{proof}
(1) Suppose $\Pi=span(X,V)$ where $X$ and $V$ are $g_R$-unitary, $V=\alpha E+Y$ and $X\perp Y\perp E$. Since $E$ is Killing, 
\begin{equation}
\nabla^R_X (d\omega)(X,Y)=2\left(g_R(\nabla^R_X\nabla^R_X E,Y)-g_R(\nabla^R_{\nabla^R_XX}E,Y)\right)=-2g_R(R^R_{EX}X,Y).\label{derivadaomega}
\end{equation}
Now, recall that $g_L$ corresponds
to the values $t=-2$ and $\varepsilon=1$ in formula (\ref{metricagt}) and therefore, applying 
theorems \ref{curvatura1}, \ref{curvatura2}
and \ref{curvatura3}, we get
\begin{align}
g_L(R^L_{VX}X,V)=g_R(R^R_{V^*X}X,V^*)+6g_R(\nabla^R_XE,Y)^2,\label{eqenteorema2}
\end{align}
and we obtain the result.\newline
(2) It follows from equation (\ref{eqenteorema2}).
\newline 
(3) Suppose $v=\alpha E+X$ with $X\in E^\perp$. Using corollary \ref{formulaRicci},
$Ric_L(E,E)=Ric_R(E,E)$. From proposition \ref{curvatura3} and formula (\ref{derivadaomega}), $Ric_L(E,X)=-Ric_R(E,X)$ and
by corollary \ref{Ricci2},
$Ric_L(X,X)=Ric_R(X,X)-2g_R(R^R_{XE}E,X)+6g_R(\nabla^R_XE,\nabla^R_XE)$. Since $E$ is Killing, 
$g_R(\nabla^R_XE,\nabla^R_XE)=g_R(R^R_{XE}E,X)$ and therefore 
\begin{equation*}
Ric_L(v,v)= Ric_R(v^*,v^*)+4g_R(\nabla^R_XE,\nabla^R_XE)
\end{equation*}
and the result follows.
\newline
(4) Since $E$ is Killing, $||A_E||^2=Ric_R(E,E)$ and thus, corollary \ref{curvaturasescalares} gives us
$S_R+2Ric_R(E,E)=S_L$. The statement holds now trivially.
\end{proof}

The existence of a Killing vector field on a Lorentzian manifold with an isolated zero implies that $M$ has even dimension, \cite{Beem}.
We can proof the following  related result.

\begin{theorem}
 Let $(M,g_L)$ be a compact Lorentzian manifold with negative sectional curvature on timelike planes. If there exists
a timelike unitary Killing vector field, then $M$ has odd dimension.
\end{theorem}
\begin{proof}
 Take $g_R=g_L+2\omega\otimes\omega$. Applying above proposition, if $\Pi$ is a plane containing $E$, then
 $K_R(\Pi)=-K_L(\Pi)>0$. Since $E$ is also unitary and 
Killing for $g_R$ (see corollary \ref{corolario1}), using Berger's theorem (\cite{Berger,Nikonorov}), $M$ has odd dimension.
\end{proof}

\begin{example}
 Since the Lorentzian Berger sphere $(\mathbb{S}^{2n+1},g_L)$ is obtained as the standard canonical variation of the
Euclidean sphere along the Hopf vector field $E$,  we have
$2\leq \mathcal{K}_E(\Pi)$ for all degenerate plane and the scalar curvature is $S_L=2n(2n+3)$.
\end{example}

We give now an application to using Bochner techniques.
Given a Killing  vector field $U$ in a compact semi-Riemannian manifold $(M,g)$, it holds 
$$
\int_M ||A_U||^2 dg=\int_M Ric(U,U) dg.
$$
Therefore, in the Riemannian case, we have
\begin{equation}
0\leq\int_M Ric_R(U,U) dg_R\label{desigualdadintegral1}.
\end{equation}
Moreover, if $U$  is nonzero everywhere and  we take $\{e_1,\ldots,e_{n-1},\frac{U}{|U|}\}$ an orthonormal basis, then
 $$
||A_U||^2=\sum_{i=1}^{n-1}g_R(\nabla^R_{e_i}U,\nabla^R_{e_i}U)+\frac{1}{g_R(U,U)}g_R(\nabla^R_UU,\nabla^R_UU)
$$
and so we can refine inequality (\ref{desigualdadintegral1}) obtaining
\begin{equation}
0\leq\int_M \left(Ric_R(U,U)-\frac{1}{g_R(U,U)}g_R(\nabla^R_UU,\nabla^R_UU)\right)dg_R.\label{desigualdadintegral2} 
\end{equation}
In the Lorentzian case, even if $U$ is timelike,  $||A_U||^2$ and $\sum_{i=1}^{n-1}g_L(\nabla^L_{e_i}U,\nabla^L_{e_i}U)$ do not have sign and 
 we can not obtain the inequalities (\ref{desigualdadintegral1}) nor (\ref{desigualdadintegral2}).
Neverthess, in \cite{Romero} it is observed that $||A_E||^2\geq0$,
where $E$ is the unitary of $U$, and thus  if $U$ is timelike and Killing it holds
\begin{equation*}
\int_M \frac{1}{g_L(U,U)}Ric_L(U,U) dg_L\leq 0.
\end{equation*}

Using the canonical variation we can prove a similar inequality to (\ref{desigualdadintegral2}) in the Lorentzian case. First, we need some preliminaries lemmas.

\begin{lemma} Let $(M,g)$ be a semi-Riemannian manifold and $E\in\mathfrak{X}(M)$ a vector field
with $g(E,E)=c\in\mathbb{R}-\{0\}$.
Suppose that $E$ is orthogonally conformal, i.e, $\left(L_E g\right)(X,Y)=2\rho g(X,Y)$ for all $X,Y\in E^\perp$,
and take $\lambda\in C^\infty(M)$. Then $U=\lambda E$ is conformal if and only if
\begin{eqnarray*}
 E(\lambda)&=&\lambda\rho,\\
cX(\lambda)&=&-\lambda g(\nabla_EE,X) \text{ for all $X\in E^\perp$}.
\end{eqnarray*}
\end{lemma}
\begin{proof} It is enough to use the formula $L_U=\lambda L_E+d\lambda\otimes\omega+\omega\otimes d\lambda$.
\end{proof}

\begin{lemma}\label{Uconforme} Let $(M,g_L)$ be a Lorentzian manifold and $U$ a timelike conformal/Killing vector field with unitary $E$.
Then $U$ is also conformal/Killing for the canonical variation along $E$.
\end{lemma}
\begin{proof}
Just apply corollary \ref{corolario1} and above lemma.
\end{proof}

\begin{theorem} Let $(M,g_L)$ be a compact Lorentzian manifold and $U$ a timelike Killing vector field. Then
$$
0\leq\int_M \left(Ric_L(U,U)-\frac{2}{g_L(U,U)}g_L(\nabla^L_UU,\nabla^L_UU)\right)dg_L.
$$
\end{theorem}
\begin{proof}
Call $E$ the unitary of $U$ ($U=\lambda E$) and consider $g_t$ the canonical variation along $E$.
For $1<t<2$, we have
$$
0\leq\int_M Ric_t(U,U) dg_t,
$$
since  $g_t$ is Riemannian and $U$ is Killing.
Using corollary \ref{formulaRicci}, 
\begin{align*}
Ric_t(U,U)&=\lambda^2 Ric_t(E,E)=\lambda^2 \left( Ric_L(E,E)-t div\nabla^L_EE+t(t-2)||A'_E||^2 \right)\\
&=Ric_L(U,U)+\frac{2t}{\lambda^2}g_L(\nabla^L_UU,\nabla^L_UU)-t div\nabla^L_UU+t(t-2)\lambda^2||A'_E||^2
\end{align*}
But $||A'_E||^2\geq0$ and $dg_t=\sqrt{|1-t|}dg_L$. Therefore 
\begin{align*}
 0\leq \int_M \left(Ric_L(U,U)-\frac{2t}{g_L(U,U)}g_L(\nabla^L_UU,\nabla^L_UU)\right)dg_L,
\end{align*}
for $1<t<2$. Taking $t\rightarrow 1$ we get the result.
\end{proof}

\begin{example}
In general, the integral (\ref{desigualdadintegral2}) is not positive in the Lorentzian case. 
In fact, 
let $(N,g_0)$ be a compact Riemannian manifold, $f\in C^\infty(N)$ a positive function and consider $(M,g_L)=\left(N\times \mathbb{S}^1,g_0-f^2 dt^2\right)$.
The vector field  $U=\partial_t$ is timelike and  Killing and it is easy to show that
\begin{equation*}
\int_M \left(Ric_L(U,U)-\frac{1}{g_L(U,U)}g_L(\nabla^L_UU,\nabla^L_UU)\right)dg_L=-\int_M g_L(\nabla f,\nabla f)dg_L \leq0.
\end{equation*}
\end{example}

\section{Standard canonical variation along a closed vector field}\label{seccionclosed}

Suppose that $E$ is a closed unitary vector field in a Riemannian manifold $(M,g_R)$
and consider $g_L=g_R-2\omega\otimes\omega$ the standard canonical variation along it.
From formula (\ref{eqdiferencianablas}),  we have	
\begin{equation*}
\nabla^L_UV=\nabla^R_UV+2g_R(\nabla^R_U E,V)E.
\end{equation*}
In particular,  the shape operators of the orthogonal leaves of
$E$ in $(M,g_R)$ and $(M,g_L)$  coincide.

\begin{proposition}Let $\left(M,g_R\right)$ be a Riemannian manifold, $E\in\mathfrak{X}(M)$ a closed unitary vector field 
 and $g_L$ the standard canonical variation along $E$.
 \begin{enumerate}
  \item If $\Pi$ is a nondegenerate plane for $g_L$, then 
  \begin{equation*}-\cos(2\theta)K_L(\Pi)=K_R(\Pi)+2\sin^2(\theta)\left(\hat{K}_R(\mathfrak{p}(\Pi))-K_R(\mathfrak{p}(\Pi))\right),\end{equation*}                                                                                                                                   
  where $\theta$ is the angle between $E$ and $\Pi$,  $\hat{K}_R$ is the induced curvature on the orthogonal leaves of $E$ and $\mathfrak{p}$ is the 
  orthogonal projection onto $E^\perp$.
  
  In particular, $K_L(\Pi)=-K_R(\Pi)$ for any plane containing $E$.
  \item If $\Pi$ is a degenerate plane for $g_L$, then 
  \begin{equation*}\mathcal{K}_L^E(\Pi)=2K_R(\Pi)+2\left(\hat{K}_R(\mathfrak{p}(\Pi))-K_R(\mathfrak{p}(\Pi))\right).\end{equation*}
  \item For any $v\in TM$, it holds 
  \begin{eqnarray*}
   Ric_L(v,v)&=&Ric_R(v,v)+2 g_R(A_E(v),v)div_R\, E\\&-&2g_R(A_E(v),A_E(v))-2g_R(R^R_{vE}E,v).\end{eqnarray*}
   \item $S_L=S_R+4E(div_R\, E)+2\left(||A_E||^2+\left(div_R\, E\right)^2\right)$.
 \end{enumerate}

\end{proposition}
\begin{proof}
 Suppose $\Pi=span(X,V)$ where $X$ and $V$ are $g_R$-unitary,  $V=\alpha E+Y$ and $X\perp Y\perp E$. Applying theorem 
\ref{curvatura1}, \ref{curvatura2}
and proposition \ref{curvatura3}, we get
$g_L(R^L_{VX}X,V)=g_R(R^R_{VX}X,V)+2\left(g_R(A_E(X),X)g_R(A_E(Y),Y)-g_R(A_E(X),Y)^2\right)$ and we can easily deduce points 1 and 2.
\newline
 (3) It follows from corollary \ref{formulaRicci}, \ref{Ricci2} and proposition \ref{curvatura3}.
 \newline 
 (4)Use that $Ric_R(E)+E(div_R\, E)+||A_E||^2=0$ and corollary \ref{curvaturasescalares}.
 \end{proof}

\begin{theorem} Let $\left(M,g_R\right)$ be a Riemannian manifold, $E\in\mathfrak{X}(M)$ a closed unitary vector field 
 and $g_L$ the standard canonical variation along $E$. If $E$ is complete and nonparallel, then there exists a point $p\in M$ such that
 $S_R(p)<S_L(p)$.
 \end{theorem}
 \begin{proof}
  Suppose that $S_L\leq S_R$ for all points in $M$. Using above proposition,
  \begin{eqnarray*}
 2E(div_R\, E)+\left(div_R\, E\right)^2 \leq 2E(div_R\, E)+||A_E||^2+\left(div_R\, E\right)^2\leq0.
  \end{eqnarray*}
Since $E$ is complete, $div_R\, E=0$ and $S_L=S_R+2||A_E||^2\geq S_R\geq S_L$. Therefore $||A_E||^2=0$ and $E$ is parallel.
\end{proof}

Now, we are interested in knowing whether the completeness is preserved by the canonical variation.  In general, 
this does not hold, as the following examples show.

\begin{example} Take $(L,g_0)$ a complete Riemann manifold and the warped product $(\mathbb{R}\times
L,-dt^2+f^2(t)g_0)$. The standard canonical variation along  $\partial_t$ is 
$g_R=dt^2+f^2(t)g_0$, which is always complete.  However, we can choose $f$ such that
$-dt^2+f^2(t)g_0$ is not complete, \cite{Sanchez}.
\end{example}
\begin{example} Consider the Lorentzian plane $(\mathbb{R}^2,g_L=-dx^2+dy^2)$.
 Call 
\begin{eqnarray*}
f(x,y)=e^{-(x+y)},\ \ \ a(x,y)=\frac{1+f(x,y)^2}{2f(x,y)}, \ \ \ b(x,y)=\frac{1-f(x,y)^2}{2f(x,y)}.
\end{eqnarray*}
Now, $E=a\partial_x+b\partial_y$ is a timelike unitary vector field and the
standard canonical  variation along it is 
$g_R=(2a^2-1)dx^2+(2b^2+1)dy^2-4ab\,dxdy$. Take $\gamma:[0,\infty)\rightarrow\mathbb{R}^2$ given by 
$\gamma(t)=(t,t)$. It escapes  any compact subset of $\mathbb{R}^2$, but $lim_{s\rightarrow\infty}\int_0^s\sqrt{g_R(\gamma',\gamma')}=\frac{\sqrt{2}}{2}$.
Therefore, $(\mathbb{R}^2,g_R)$ is incomplete.
\end{example}

An important case when completeness is preserved is when $E$ is a Killing  vector field. In this case, 
if $(M,g_R)$ is complete, then it is also $(M,g_L)$.
Using this, it follows that a compact Lorentzian manifold furnished with a timelike Killing vector field is complete, \cite{Romero2}.

We can prove another case when completeness is assured. For this, we give the following definition.

\begin{definition} Let $\overline{M}$ be a nondegenerate hypersurface of a semi-Riemannian  manifold $(M,g)$. We say that $\overline{M}$ 
is strongly curved if its shape operator is semi-definite.
\end{definition}

\begin{proposition} Let $(M,g_R)$ be a Riemannian manifold and $E$ a complete, unitary and 
closed vector field. If the
orthogonal leaves of $E$ are complete and  strongly curved, then $(M,g_R)$ is complete.
\end{proposition}
\begin{proof} We can suppose that $M$ is simply connected. Hence, it
splits as $\big(\mathbb{R}\times L,ds^2+g_s\big)$, where
$\partial_s$ is identified with $E$, $L$ is an
orthogonal leaf and $g_s$ is a Riemannian metric on $L$ for each $s\in\mathbb{R}$, \cite{Gut}.  

Suppose that  the shape operator of the orthogonal leaves 
is semi-definite negative.
Let $v\in T_x L$ and take 
$V(s)=(0_s,v_{x})$. Then, since $[V,E]=0$, we have $\frac{d}{ds}g_R(V(s),V(s))=2g_R(\nabla^R_V E,V)\geq
0$. Thus $g_s(v,v)$ is an non decreasing
function.

Call $d$ the distance induced by $g_R$ in $M$, $d_s$ the distance induced by $g_s$ in $L$ and 
take $\{p_n=(t_n,x_n)\}$ a Cauchy sequence. Since 
$d(p_n,p_m)\geq |t_n-t_m|$ we have that $\{t_n\}$ converges to, say, $t_0$ and we can suppose that 
$|t_n-t_0|\leq\delta$ for all $n\in\mathbb{N}$ and certain $\delta\in\mathbb{R}$.
Given $0<\varepsilon<\delta$ there is $n_0\in\mathbb{N}$ such that $d(p_m,p_n)<\varepsilon$ for $m,n\geq n_0$.
Let $\gamma(s)=(s(t),x(t))$ be a curve with $\gamma(0)=p_m$ and $\gamma(1)=p_n$. We can suppose that $|s(t)-t_0|\leq 2\delta$
since on the contrary case, $L(\gamma)>\delta>d(p_m,p_n)$. Now,
\begin{eqnarray*}
L(\gamma)&=&\int_{0}^{1}\sqrt{s'(t)^2+g_{s(t)}(x'(t),x'(t))}\geq\int_{0}^{1}\sqrt{g_{s^*}(x'(t),x'(t))}\geq d_{s^*}(x_n,x_m),\\
\end{eqnarray*}
where $s^*=t_0-2\delta$ and therefore $d_{s^*}(x_n,x_m)\leq d(p_n,p_m)<\varepsilon$. Hence $\{x_n\}$ is a Cauchy
sequence in $(L,g_{s^*})$ and so it converges.
\end{proof}

\begin{theorem}
 Let $(M,g_L)$ a Lorentzian manifold and $E$ a complete timelike unitary and closed vector
 field. If the orthogonal leaves of $E$ are complete and strongly curved, then the standard canonical
variation $g_R$ along $E$ is also complete.
\end{theorem}
\begin{proof}
As it was said at the beginning of this section, the shape operators coincide in $(M,g_L)$ and $(M,g_R)$, thus
orthogonal leaves of $E$ are strongly curved in $(M,g_R)$ too. Applying above proposition
we obtain that $g_R$ is complete.
\end{proof}

\section{Lightlike hypersurfaces}
Let $(M,g_L)$ be a Lorentzian manifold, $E\in\mathfrak{X}(M)$ a timelike unitary vector field and $\overline{M}$
 a lightlike hypersurface. We can
fix a lightlike vector field $\xi\in\mathfrak{X}(\overline{M})$ with $g_L(E,\xi)=\frac{1}{\sqrt{2}}$ and consider 
the screen distribution given by
$\mathcal{S}=T\overline{M}\cap E^\perp$. Given $U,V\in\mathfrak{X}(\overline{M})$, 
the second fundamental form of $\overline{M}$ is $B(U,V)=-g_L(\nabla^L_U\xi,V)$. 
The hypersurface $\overline{M}$ is totally geodesic if $B=0$ and totally umbilic if $B=\rho g$ for 
certain $\rho\in C^\infty(\overline{M})$. 

We can decompose 
\begin{eqnarray*}
\nabla^L_U V&=&\overline{\nabla}^L_U V+B(U,V)N, \\
\nabla^L_U\xi&=&-\tau(U)\xi-A^*(U),
\end{eqnarray*}
where $\overline{\nabla}^L_UV\in T\overline{M}$, $A^*(U)\in \mathcal{S}$, $\tau$ is a one form and $N=\sqrt{2}E+\xi$ is the transverse vector field to 
$\overline{M}$. Recall that $\nabla^L_\xi \xi=-\tau(\xi)\xi$ and $N$ is lightlike and orthogonal to $\mathcal{S}$.
Moreover, $\tau(\xi)=\sqrt{2} g_L(\nabla^L_\xi E,\xi)$.

Observe that $X_0=\frac{1}{\sqrt{2}}E+\xi$ is orthogonal to $E$ and therefore $\xi=-\frac{1}{\sqrt{2}}E+X_0$ and $N=\frac{1}{\sqrt{2}}E+X_0$, i.e.,
$N$ is the symmetric of $-\xi$ respect to $E$.

If we consider $g_R$ the standard canonical variation, then $g_R(\xi,\xi)=g_R(N,N)=1$ and  
$N$ is $g_R$-normal to $\overline{M}$. Call $\mathbb{I}$ the second fundamental
form of $\overline{M}$ in $(M,g_R)$.

\begin{proposition} \label{propssf}Given $X,Y\in \mathcal{S}$, it holds
\begin{eqnarray*}
\mathbb{I}(X,Y)&=&\left( B(X,Y)-\frac{1}{\sqrt{2}} \left(L_Eg_L\right)(X,Y)\right)N,\\
\mathbb{I}(X,\xi)&=&-g_L(\nabla^L_{E+\sqrt{2}\xi}E,X)N,\\
\mathbb{I}(\xi,\xi)&=&-\big(2g_L(\xi,\nabla^L_EE)+\tau(\xi)\big)N.
\end{eqnarray*}

\end{proposition}
\begin{proof}
By proposition \ref{diferencianablas},
 $g_R(D(X,Y),N)=-\frac{1}{\sqrt{2}}\left(L_Eg\right)(X,Y)$ and thus we have  
\begin{eqnarray*}g_R(\mathbb{I}(X,Y),N)=g_R(\nabla^L_XY,N)-\frac{1}{\sqrt{2}}\left(L_Eg\right)(X,Y)=
B(X,Y)-\frac{1}{\sqrt{2}}\left(L_Eg\right)(X,Y).
\end{eqnarray*}
On the other hand $g_R(D(X,\xi),N)=-\sqrt{2}g_L(\nabla^L_{\xi} E,X)-g_L(\nabla^L_EE,X)$ 
and since $g_R(\nabla^L_X\xi,N)=0$, we get  $\mathbb{I}(X,\xi)=-g_L(\nabla^L_{E+\sqrt{2}\xi}E,X)N$.

Finally, 
$g_R(D(\xi,\xi),N)=-\sqrt{2}g(\xi,\nabla^L_N E)=-\big(2g_L(\xi,\nabla^L_EE)+\tau(\xi)\big)$,
and since $g_R(\nabla^L_\xi\xi,N)=0$, we have
$\mathbb{I}(\xi,\xi)=-\big(2g_L(\xi,\nabla^L_EE)+\tau(\xi)\big)N$.
\end{proof}

If $E$ is parallel and $\overline{M}$ is totally geodesic, then it is also totally geodesic
in $(M,g_R)$, but however, if $\overline{M}$ is umbilic, then it is not  umbilic in $(M,g_R)$.

Recall that the lightlike mean curvature of $\overline{M}$ is given by $H_L=\sum_{i=1}^{n-2}B(e_i,e_i)$ where
$\{e_1,\ldots,e_{n-2}\}$ is an orthonormal basis of $\mathcal{S}$ and the mean curvature of $\overline{M}$
 as hypersurface of $(M,g_R)$ by 
$H_R=\sum_{i=1}^{n-1} g_R(\mathbb{I}(v_i,v_i),N)$ being $\{v_1,\ldots,v_{n-1}\}$ a $g_R$ orthonormal basis
of $T\overline{M}$.

\begin{proposition}\label{curvaturamedia} Let $(M,g_L)$ be a
 Lorentzian manifold, $E\in\mathfrak{X}(M)$ a timelike unitary vector field and 
$g_R$  the standard canonical variation along it. Take 
$\overline{M}$ a lightlike hypersurface. 
\begin{itemize}
 \item $H_R=H_L-\sqrt{2}div_L E+\tau(\xi)$.
\item If $\overline{M}$ is compact and orientable, then $\int_{\overline{M}} H_L dg_R=0$.
\end{itemize}
\end{proposition}
\begin{proof}
If $\{e_1,\ldots,e_{n-2}\}$ an orthonomal basis in $\mathcal{S}$, then $\{e_1,\ldots,e_{n-2},\xi\}$ is a $g_R$-orthonormal
basis of $\overline{M}$. Therefore, using proposition \ref{propssf},
\begin{eqnarray*}
H_R=
H_L-\sqrt{2}\sum_{i=1}^{n-2}g_L(\nabla^L_{e_i}E,e_i)-2g_L(\xi,\nabla^L_EE)-\tau(\xi).
\end{eqnarray*}
But now, observe that $\{e_1,\ldots,e_{n-2},\sqrt{2}\xi+E,E\}$ is an orthonormal basis and so 
\begin{eqnarray*}
div_L E&=&\sum_{i=1}^{n-2}g_L(\nabla^L_{e_i}E,e_i)+g_L(\nabla^L_{\sqrt{2}\xi+E}E,\sqrt{2}\xi+E)\\
&=&\sum_{i=1}^{n-2}g_L(\nabla^L_{e_i}E,e_i)+\sqrt{2}\tau(\xi)+\sqrt{2}g_L(\nabla^L_EE,\xi).
\end{eqnarray*}

For the second point, since  $g_R(D(e_i,\xi),e_i)=0$ (corollary \ref{corolario1}),
we have $div^{\overline{M}}_R\xi=-H_L$.
\end{proof}

We can easily obtain the following 
generalization of theorem 8 in \cite{Bejancu}.
\begin{corollary}
 Let $(M,g_L)$ be a Lorentzian manifold, $E$ a timelike unitary Killing vector field and $\overline{M}$ a lightlike hypersurface.
Consider  $g_R$ the 
standard canonical variation along $E$.
The lightlike mean curvature of $\overline{M}$ vanishes if and only if the mean curvature of $\overline{M}$ as hypersurface of $(M,g_R)$ vanishes.
\end{corollary}

In a time-orientable and orientable Lorentzian manifold satisfying the null convergence condition ($Ric(u,u)\geq0$ for 
all lightlike vector $u$) any compact lightlike hypersurface with nonpositive (or nonnegative) mean curvature is totally geodesic. 
In fact, we only have to take into account second point of proposition \ref{curvaturamedia} and the well-known 
Raychaudhuri equation (see for example \cite{Duggal, Gallo, Kupeli}).
An standard application of the Raychaudhuri equation also gives us that a lightlike hypersurface is totally geodesic if  the null convergence condition and the 
completeness of lightlike geodesics of the hypersurface are assumed. Observe that this result is 
not applicable to above situation because the compactness 
does not ensure, in general, the completeness of the lightlike geodesics of the hypersurface.

We can obtain other results ensuring that a compact lightlike hypersurface is totally geodesic under
curvature hypotheses.
 First, we need the following lemma.

\begin{lemma}\label{lemaortogonalkilling}
 Let $(M,g_L)$ be a Lorentzian manifold and $U\in\mathfrak{X}(M)$ a timelike Killing vector field with unitary $E$. If $\lambda=|U|$, then
for any $X\perp E$ we have 
\begin{equation}
 g_L(R^L_{XE}E,X)=\frac{1}{\lambda}g_L(\nabla_X^L\nabla\lambda,X)+g_L(\nabla^L_X E,\nabla^L_X E).
\end{equation}
\end{lemma}

If $U$ is a Killing vector field and $\overline{M}$ a compact hypersurface of a Riemannian manifold
$(M,g_R)$  with normal unitary $N$, then
\begin{equation}\label{formulaintegral1}
 \int_{\overline{M}}g_R(U,N)H_R dg_R=0.
\end{equation}
Moreover, if $\overline{M}$ has constant mean curvature, then
\begin{equation}\label{formulaintegral2}
\int_{\overline{M}}g_R(U,N)||S||^2-Ric_R(N,\overline{U}) dg_R=0, 
\end{equation}
where $S(X)=-\nabla^R_X N$ is the shape operator of $\overline{M}$ and $\overline{U}$ is the $g_R$-projection of $U$
onto $\overline{M}$, \cite{Yano}.

\begin{theorem}
 Let $(M,g_L)$ be a Lorentzian manifold, $U\in\mathfrak{X}(M)$ a timelike Killing vector field and 
$\overline{M}$ a compact lightlike hypersurface. Take $E$ the unitary of $U$, $\lambda=|U|$
and $\xi\in\mathcal{X}(\overline{M})$ a lightlike vector field with $g_L(\xi,E)=\frac{1}{\sqrt{2}}$. Call $X_0$ the orthogonal projection of $\xi$ onto $E^\perp$ and
$N$ the symmetric of $-\xi$ respect to $E$.
If 
\begin{itemize}
 \item $H_L-g_L(\nabla\ln\lambda,\xi)$ has sign.
\item $0\leq Ric_L(N,\xi)+\vartriangle\ln\lambda-4g_L(\nabla^L_{X_0}E,\nabla^L_{X_0}E)$.
\end{itemize}
Then, $\overline{M}$ is totally geodesic.
\end{theorem}
\begin{proof} First, observe that $\nabla^L_EE=\nabla\ln\lambda$ and $\tau(\xi)=-g_L(\xi,\nabla\ln\lambda)$.
 Take $g_R$ the canonical variation along $E=\frac{U}{\lambda}$. We know that $U$ is also Killing for $g_R$ (lemma \ref{Uconforme})
and it is easy to show that $g_R(U,N)=\frac{\lambda}{\sqrt{2}}$. Applying formula
(\ref{formulaintegral1}), we have $\int_{\overline{M}}\lambda H_R dg_R=0$. But, from proposition \ref{curvaturamedia},
$H_R=H_L-g_L(\nabla\ln\lambda,\xi)$ which has sign. Therefore, $H_R=0$ and
since the projection of $U$ onto $\overline{M}$ is given by $-\frac{\lambda}{\sqrt{2}}\xi$, 
formula (\ref{formulaintegral2}) gives us
$$
\int_{\overline{M}}\lambda ||S||^2 dg_R=-\int_{\overline{M}}\lambda Ric_R(N,\xi) dg_R.
$$
Using corollary \ref{formulaRicci} and \ref{Ricci2} and lemma \ref{lemaortogonalkilling},
\begin{eqnarray*} 
Ric_R(N,\xi)&=&-\frac{1}{2}Ric_R(E,E)+Ric_R(X_0,X_0)\\
&=&Ric_L(N,\xi)+\vartriangle\ln\lambda-6g_L(\nabla^L_{X_0}E,\nabla^L_{X_0}E)\\
&+&2\left( g_L(R^L_{X_0 E}E,X_0)-\frac{1}{\lambda}
g_L(\nabla^L_{X_0}\nabla\lambda,X_0)\right) \\
&=&Ric_L(N,\xi)+\vartriangle\ln\lambda-4g_L(\nabla^L_{X_0}E,\nabla^L_{X_0}E).
\end{eqnarray*}
Thus, $||S||^2=0$ and from proposition \ref{propssf}, $\overline{M}$ is a totally geodesic lightlike hypersurface of 
$(M,g_L)$.
\end{proof}

\begin{corollary} Let $(M,g_L)$ be a Lorentzian manifold, $E\in\mathfrak{X}(M)$ a timelike unitary Killing
 vector field and $\overline{M}$ a compact lightlike hypersurface. 
 Take $\xi\in\mathcal{X}(\overline{M})$ a lightlike vector field with $g_L(\xi,E)=\frac{1}{\sqrt{2}}$ and $N$ the symmetric of $-\xi$ respect to $E$.
 If 
\begin{itemize}
 \item $H_L$ has sign.
\item $0\leq Ric_L(N,\xi)+2 K_L(span(\xi,N))$.
\end{itemize}
Then $\overline{M}$ is totally geodesic.
\end{corollary}
\begin{proof}
 From lemma \ref{lemaortogonalkilling}, $g_L(\nabla^L_{X_0}E,\nabla^L_{X_0}E)=g_L(R^L_{X_0 E}E,X_0)=-\frac{1}{2}K^L(span(\xi,N))$.
\end{proof}


\begin{thebibliography}{99}

\bibitem{Antidogbe} C. Antidogb\'e, Scalar curvature on lightlike hypersurfaces, Appl. Sci. 11 (2009) 9-18.

\bibitem{Beem} J.K. Beem, P.E. Ehrlich and S. Markvorsen, Timelike isometries and Killing vector fields, Geom. Dedicata 26 (1988) 247-258.

\bibitem{Bejancu} A. Bejancu, A. Ferr\'andez and P. Lucas, A new viewpoint on geometry of a lightlike hypersurface in a semi-Euclidean space, Saitama Math. J. 16 (1998) 31-38.

\bibitem{Berger} M. Berger, Trois remarques sur les vari\'et\'es Riemanniennes a courbure positive, C. R. Acad. Sci. Paris 263 (1966) 76-78.


\bibitem{Besse}A. Besse, Einstein manifolds, Springer-Verlag, Berlin, 1987.


\bibitem{Popescu} V. Boju and M. Popescu, Espaces à courbure quasi-constante, J. Diff. Geom. 13 (1978) 373-383.


\bibitem{Duggal} K.L. Duggal and A. Gim\'enez, Lightlike hypersurfaces of Lorentzian manifolds with distinguished
screen distribution, J. Geom. Phys. 55 (2005) 107-122.



\bibitem{Flores} J.L. Flores, M.A. Javaloyes and P. Piccione, Periodic geodesics and geometry of compact Lorentzian manifolds with a Killing 
vector field, Math. Z. 267 (2011) 221-233.

\bibitem{Ganchev} G. Ganchev and V. Mihova, Riemannian manifolds of quasi-constant sectional curvatures, J. reine angew. Math. 522 (2000), 119-141.




\bibitem{Nikonorov} V.N. Berestovski and Y.G. Nikonorov, Killing vector fields of constant length on Riemannian manifolds, Sib. Math. J. 49 (2008) 395-407.

\bibitem{Hawking} S.W. Hawking and G.F.R. Ellis, The large scale structure of space-time, Cambridge University Press, London-New York, 1973.

\bibitem{Gallo} G.J. Galloway, Maximum principles for null hypersurfaces and null splitting theorems, Ann. Henri Poincar\'e 1 (2000) 543-567.



\bibitem{Gut} M. Gutierrez and B. Olea, Global decomposition of a Lorenzian manifold as a Generalized Robertson-Walker space, Diff. Geom. Appl. 27 (2009) 146-156.


\bibitem{Gut2} M. Guti\'errez, F.J. Palomo and A. Romero, A Berger-Green type inequality for compact Lorentzian manifolds, Trans. Amer. Math. Soc. 354 (2002) 4505-4523.

\bibitem{Harris}S. G. Harris, A triangle comparison theorem for Lorentz
manifolds, Indiana Univ. Math. J. 31 (1982) 289-308.

\bibitem{Karcher}H. Karcher, Infinitesimale Charakterisierung von Friedmann-Univeren, Arch. Math. 38 (1982) 5864.

\bibitem{Kupeli}  D.N. Kupeli, On null submanifolds in spacetimes, Geom. Dedicata 23 (1) (1987) 33–51.

\bibitem{Romero} A. Romero and M. S\'anchez, Bochner's techniques in Lorentzian manifold, Pacific J. Math. 186 (1998) 141-148.


\bibitem{Romero2} A. Romero and M. S\'anchez, Completeness of compact Lorentz manifolds admiting a timelike conformal Killing field, Proc. Amer. Math. Soc. 123 (1995) 2831-2833.

\bibitem{Sanchez} M. S\'anchez, On the geometry of Generalized
Robertson-Walker Spacetimes: Geodesics, Gen. Relativity Gravitation 30 (1998) 915-932.

\bibitem{Sharma1} V.V. Reddy, R. Sharma and S. Sivaramakrishnan, Lorentzian metric
 induced from a background Riemannian metric, 
 Int. J. Pure Appl. Math. 47 (2008) 343-351.

\bibitem{Sharma2} V.V. Reddy, R. Sharma and S. Sivaramakrishnan, Spacetimes through Hawking–Ellis construction with a background Riemannian metric,
 Class. Quantum Grav. 24 (2007) 3339-3345.

\bibitem{Strake} M. Strake, Curvature increasing metric variation, Math. Ann. 276 (1987) 633-641.

\bibitem{Yano} K. Yano, Closed hypersurfaces with constant mean curvature in a Riemannian
manifold, J. Math. Soc. Japan 17 (1995) 333-340.

\bibitem{Yurtsever} U. Yurtsever, Test fields on compact spacetimes, J. Math. Phys. 31 (1990) 3064-3078.

\end{thebibliography}
\end{document}